\documentclass[a4paper,11pt]{article}
\usepackage{url}
\usepackage[T1]{fontenc} 
\usepackage{amsmath,amsthm}
\usepackage[francais,english]{babel}
\usepackage[dvips]{graphicx}
\usepackage{todonotes}
\usepackage{color}
\usepackage{tikz}
\usetikzlibrary{arrows,positioning,shapes,matrix,backgrounds}
\usetikzlibrary{fit,calc,decorations.pathreplacing}

\usepgflibrary{arrows}
\usepackage{soul}
\usepackage{amsfonts,amssymb}
\usepackage{geometry}
\usepackage{hyperref}
\usepackage{stmaryrd}
\usepackage{fancyhdr}
\usepackage{url}
\usepackage{todonotes}
\usepackage{dsfont}
\usepackage[utf8]{inputenc} 
\usepackage{enumerate}

\theoremstyle{definition}

\newtheorem*{thm*}{Theorem}

\newtheorem{prop}{Proposition}[section]

\newtheorem{lemma}[prop]{Lemma}

\newtheorem{thm}[prop]{Theorem}

\newtheorem{notation}[prop]{Notation}
\newtheorem{corollary}[prop]{Corollary}

\newtheoremstyle{pourlesremarques}{\topsep}{\topsep}{\normalfont}{}{\bfseries}{.}{ }{}
\theoremstyle{pourlesremarques}

\newtheorem*{rem*}{Remark}
\newtheoremstyle{pourlesexemples}{\topsep}{\topsep}{\normalfont}{}{\bfseries}{.}{ }{}
\theoremstyle{pourlesexemples}

\def\presuper#1#2%
  {\mathop{}%
   \mathopen{\vphantom{#2}}^{#1}%
   \kern-\scriptspace%
   #2}

\newcommand{\Ad}{{\operatorname{Ad}}}
\newcommand{\lZ}{\mathcal{Z}}

\def\Irr{\operatorname{Irr}}

\newcommand{\Ind}{\operatorname{Ind}}

\newcommand{\ind}{\operatorname{ind}}
\newcommand{\K}{\mathrm{K}}

\newcommand{\Hom}{\operatorname{Hom}}

\newcommand{\G}{\mathrm{G}}

\newcommand{\dist}{\operatorname{-dist}}

\newcommand{\Res}{\operatorname{Res}}

\newcommand{\Ker}{\mathrm{Ker}}

\newcommand{\w}{\varpi}
\renewcommand{\d}{\delta}
\newcommand{\R}{\mathbb{R}}

\newcommand{\V}{\mathrm{V}}
\newcommand{\C}{\mathbb{C}}

\newcommand{\Ze}{\mathrm{Z}}

\newcommand{\op}{\mathrm{op}}
\newcommand{\Q}{\mathbb{Q}}

\renewcommand{\H}{\mathrm{H}}
\renewcommand{\L}{\mathrm{L}}
\newcommand{\Z}{\mathbb{Z}}

\newcommand{\1}{\mathbf{1}}

\newcommand{\U}{\mathrm{U}}

\def\lN{\mathcal{N}}
\def\\Hom{\operatorname{\Hom}}
\def\Irr{\operatorname{Irr}}

\def\dim{\operatorname{dim}}

\def\Vect{\operatorname{Vect}}

\def\leq{\leqslant}
\def\geq{\geqslant}

\def\Res{\operatorname{Res}}

\def\presuper#1#2%
  {\mathop{}%
   \mathopen{\vphantom{#2}}^{#1}%
   \kern-\scriptspace%
   #2}
\makeatletter
\DeclareRobustCommand{\rvdots}{%
  \vbox{
    \baselineskip4\p@\lineskiplimit\z@
    \kern-\p@
    \hbox{.}\hbox{.}\hbox{.}  \hbox{.}
  }}
\makeatother
\makeatletter
\newcommand{\address}[1]{\gdef\@address{#1}}
\newcommand{\email}[1]{\gdef\@email{\url{#1}}}
\newcommand{\@endstuff}{\par\vspace{\baselineskip}\noindent\small
\begin{tabular}{@{}l}\scshape\@address\\\textit{E-mail address:} \@email\end{tabular}}
\AtEndDocument{\@endstuff}
\makeatother

\title{Distinction for unipotent $p$-adic groups}
\author{Nadir Matringe}

\address{Universit\'e de Poitiers, Laboratoire de Math\'ematiques et Applications,\\
T\'el\'eport 2 - BP 30179, Boulevard Marie et Pierre Curie,\\
86962, Futuroscope Chasseneuil Cedex. France.} 
\email{nadir.matringe@math.univ-poitiers.fr}
\begin{document}
\maketitle

\begin{abstract}
Let $F$ be a $p$-adic field and $\boldsymbol{\U}$ be a unipotent group defined over $F$, and set 
$\U=\boldsymbol{\U}(F)$. Let $\sigma$ be an involution of $\boldsymbol{\U}$ defined over $F$. Adapting the arguments of Yves Benoist (\cite{B1}, \cite{B2}) in the real case, we prove the following result: an irreducible representation $\pi$ of $\U$ is $\U^{\sigma}$-distinguished if and only if it is $\sigma$-self-dual and in this case $\Hom_{\U^\sigma}(\pi,\C)$ has dimension one. 
When $\sigma$ is a Galois involution these results imply a bijective correspondence between the set $\Irr(\U^\sigma)$ of isomorphism classes of irreducible representations of 
$\U^\sigma$ and the set $\Irr_{\U^\sigma-\mathrm{dist}}(\U)$ of isomorphism classes of distinguished irreducible representations of $\U$. 
\end{abstract}

\section{Introduction}

Let $\mathbf{G}$ be a connected algebraic group defined over a field $F$, and $\sigma$ be an $F$-rational involution of $\G$. One says that a complex representation $\pi$ of $\G=\mathbf{G}(F)$ is distinguished if $\Hom_{\G^\sigma}(\pi,\C)\neq 0$. One is in general interested in computing the dimension of $\Hom_{\G^\sigma}(\pi,\C)$ when $\pi$ is irreducible, as well understanding the relation between irreducible distinction and conjugate self-duality.

One extensively studied situation is that of distinction by a Galois involution. Let $E/F$ be a separable extension of quadratic field, take 
$\mathbf{G}=\Res_{E/F}(\mathbf{H})$ for $\mathbf{H}$ be a connected algebraic group defined over $F$. Then $\sigma$ is taken to be the corresponding Galois involution. A case of interest is that 
of finite fields, in which case 
it has been shown in \cite[Theorem 2]{Pcomp} that an irreducible representation $\pi$ of $\G$ which is stable is distinguished if and only if it is conjugate self-dual: $\pi^\vee\simeq \pi^\sigma$. 

 The question of the 
relation between distinction and conjugate self-duality as well as that of the dimension of $\Hom_{\H}(\pi,\C)$ remains interesting for smooth representations when $F$ is $p$-adic, and it has attracted a lot of attention when 
$\mathbf{G}$ is reductive. The answer is not known in general, but a conjectural and very precise answer in terms of Langlands parameters is provided by \cite{Pconj}. 
It in particular roughly says that if $\pi$ is an irreducible distinguished (by a certain quadratic character) representation of $\G$, 
then $\pi^\vee$ and $\pi^\sigma$ should be in the same $\L$-packet, and moreover there should be a correspondence 
between irreducible distinguished representations 
of $\G$ and irreducible representations of $\mathbf{H}^{\op}(F)$ where the opposition group $\mathbf{H}^{\op}$ 
is a certain reductive group defined over $F$ and isomorphic to $\mathbf{H}$ over $E$. 

Going back to a general involution, still with $F$ a $p$-adic field, it seems that such questions have not attracted as much attention when $\mathbf{G}$ is unipotent. 
It turns out that the different answers, provided by this paper, are simple as well as their proofs. In fact they were completely solved when $F=\R$ by Y. Benoist in \cite{B1} and
\cite{B2}, 
where moreover a Plancherel formula for the corresponding symmetric space was established. Our results are the same, and the proofs are very close though sometimes the arguments have to be different. Let us quickly describe the content of this note.

If $\boldsymbol{\G}=\boldsymbol{\U}$ is unipotent, then a smooth irreducible representation of $\U=\boldsymbol{\U}(F)$ is distinguished if and only if it is conjugate self-dual, in which case 
$\Hom_{\U^\sigma}(\pi,\C)$ has dimension one (Proposition \ref{prop mult 1} and Theorem \ref{thm distinction and sel-duality}). Moreover when $\sigma$ is a Galois involution, there is a bijective correspondence 
between distinguished irreducible representations of $\U$ and representations of 
$\U^\sigma$ (Corollary \ref{cor correspondence}). Hence, setting $\mathbf{H}=\mathbf{U}$, in a certain sense 
$\boldsymbol{\mathbf{U}}^{\op}=\boldsymbol{\mathbf{U}}$ when $\mathbf{U}$ is unipotent.

As in \cite{B1} and \cite{B2}, all proofs are based on the Kirillov construction and parametrization (\cite{K}, \cite{VD}) of irreducible representations of $\U$. In fact as the 
Kirillov construction in the case of smooth irreducible representations of $p$-adic fields seems not to be fully written in details in the litterature, 
we do this work in Section \ref{section kirillov} for the convenience of the reader. Note that the classification for continuous irreducible unitary representations of $\U$ 
on Hilbert spaces is available in several papers (see \cite{EK} and the references there), so that the Kirillov classification of smooth irreducible unitary representations of $\U$ can certainly be deduced 
from it by conisdering the injection of this category into that of irreducible unitary representations on Hilbert spaces by taking smooth vectors (though 
we could not find a proof of this result in the case at hand), but in any case we give a direct proof here, for which we claim no originality other 
than that we did not find it written as such in the litterature. We make use of a result of \cite{VD}, which is 
very well suited to obtain Kirillov's classification in a quick manner.

\section{Notations}

In this paper $F$ is a $p$-adic field, i.e. a finite extension of $\Q_p$, with ring of integers $O_F$ and uniformizer $\w_F$. We consider $\mathbf{U}$ a (necessarily connected) unipotent 
group defined over $F$. We denote by $\mathbf{U}$ a connected 
unipotent group defined over $F$ with 
Lie algebra $\boldsymbol{\mathcal{N}}$ so that \[\exp:\boldsymbol{\mathcal{N}}\rightarrow \mathbf{U}\] is an isomorphism of algebraic $F$-varieties with reciprocal map $\ln$ 
(\cite[Proposition 4.1]{DG}).

We set 
$\U=\mathbf{U}(F)$ and $\lN=\boldsymbol{\mathcal{N}}(F)$, the map $\exp$ restricts as a homeomorphism from 
$\lN$ to $\U$. We will say that $\U'$ is an $F$-subgroup of $\U$ if it is the $F$-points of a closed algebraic subgroup $\mathbf{U'}$ of $\mathbf{U}$ defined over $F$. The map $\exp$ 
induces a bijection between Lie sub-algebras of $\lN$ (resp. $\mathbf{\lN}$) and the $F$-subgroups of $\U$ (resp. $\mathbf{U}$), for which ideals correspond to normal subgroups. Moreover if $\mathbf{U}'$ is an $F$-subgroup of 
$\mathbf{U}$ then $\U/ \U'\simeq (\mathbf{U}/\mathbf{U'})(F)$ by \cite[14.2.6]{Spr}, and this bijection becomes a group isomorphism if $\mathbf{U'}$ is normal in $\mathbf{U}$ in which case both quotients identify to 
$\lN/\lN'= (\boldsymbol{\lN}/\boldsymbol{\lN'})(F)$ via $\exp$.  

We denote by $\lZ$ the center of $\lN$, and by $\Ze$ the center $\exp(\mathcal{Z})$ of $\U$.

As a convention if $\U_i$ or $\U'$ is an $F$-subgroup of $\U$, we will denote by $\lN_i$ or $\lN'$ its Lie algebra. 

A fundamental example of unipotent group is the Heisenberg group \[\U=\H_3=\{h(x,y,z):=\begin{pmatrix} 1 & x & z \\ & 1 & y\\ & & 1\end{pmatrix}, \ x, \ y, \ z \in F\}.\] 
We will denote by \[\L=\{h(0,y,z), \ y, \ z \in F\}\] its normal Lagrangian subgroup $\H_3$.

We denote by $\Irr(\U)$ the set of isomorphism classes of (always smooth) irreducible representations of 
$\U$ and by $\Irr_{\U^\sigma-\mathrm{dist}}(\U)$ the subset of isomorphism classes of distinguished irreducible representations of $\U$. For $\pi\in \Irr(\U)$ we will write $c_\pi$ its central character. We will say that a representation is unitary if it preserves a positive definite hermitian form. We write $\ind$ for compact induction and $\Ind$ for induction 
(in our situation normalized induction will coincide with non-normalized induction). We recall that if $\pi'$ is a smooth representation of a closed subgroup $\U'$ of $\U$, then if $\ind_{\U'}^{\U}(\pi)$ is admissible we have 
$\ind_{\U'}^{\U}(\pi)=\Ind_{\U'}^{\U}(\pi)$.

\section{The Kirillov classification}\label{section kirillov}

\subsection{Definitions}

In this section we fix $\psi:F\rightarrow \C_u$ a non trivial character. Take $\phi\in \lN^*=\Hom_{F}(\lN,F)$ and let $\lN'$ be a Lie sub-algebra of 
$\lN$, we will say that the pair $(\phi,\lN')$ is \emph{polarized} for $\lN$ if $\lN'$ is a totally isotropic sub-space of maximal dimension for the $F$-bilinear form
\[B_\phi:\lN\times \lN\rightarrow F\] defined by \[B_\phi(X,Y)=\phi([X,Y]).\] 
We denote by $\mathcal{P}(\lN)$ the set of polarized pairs for $\lN$. 
The group $\U$ acts on $\mathcal{P}(\lN)$ by the formula 
\[u.(\phi,\lN')=(\phi\circ \Ad(u)^{-1}, \Ad(u)(\lN')).\] More generally it acts by the same formula on the set of 
pairs $(\phi,\lN')$ where $\phi$ is a linear form on $\lN$ and $\lN'$ is a Lie sub-algebra (or even a vector subspace) of $\lN$.

Whether $(\phi,\lN')$ is polarized or not, as soon as $\lN'$ is totally isotropic for $B_\phi$, the linear form $\phi$ defines a character 
$\psi_\phi$ of $\U':=\exp(\lN')$  
given by \[\psi_\phi(u')=\psi(\phi(\ln(u'))).\] 
We set \[\pi(\U',\U,\psi_\phi)=\ind_{\U'}^{\U}(\psi_\phi).\] Note that 
if $(\phi,\lN')$ and $(\phi',\lN'')$ are in the same $\U$-orbit, then the inducing data 
$(\psi_\phi,\U')$ and $(\psi_{\phi'},\U'')$ are conjugate and 
\[\pi(\U',\U,\psi_\phi)\simeq \pi(\U'',\U,\psi_{\phi'}).\]
 The author of \cite{VD} notices in \cite[Section 6]{VD} 
that the results of \cite{K} on unitary representations of real unipotent groups apply with the same proofs to 
unitary representations (acting on Hilbert spaces) of unipotent $p$-adic groups. They also apply to smooth representations of unipotent $p$-adic groups with the same proofs. For the sake of completeness we will recall the proofs, using handy results from \cite[Proof of Theorem 4]{VD}. 

\subsection{Preparation}

In this paragraph we suppose that $\lZ$ is of dimension $1$. By Kirillov's lemma (\cite[Lemma 4.1]{K}) there is a "canonical" decomposition
\[\lN=F.X\oplus F.Y\oplus F.Z \oplus W\] which means that the vectors 
$X$, $Y$, $Z$ and the $F$-vector space $W$ have the following properties:

\begin{enumerate}
\item $\lZ= F.Z$.
\item $[X,Y]=Z$.
\item $[Y,W]=\{0\}$.
\end{enumerate}

The Lie sub-algebra \[\lN_0:=F.Y \oplus F.Z \oplus W\] is automatically a codimension $1$ ideal of $\lN$ and we set \[\U_0=\exp(\lN_0).\] 

Note that $\Vect_F(X,Y,Z)$ is a Lie algebra isomorphic to that of $\H_3$, hence $\exp(\Vect_F(X,Y,Z))$ is a closed subgroup of $\U$ isomorphic to $\H_3$. We set 
\[h(x,y,z)=\exp(y.Y)\exp(x.X)\exp(z.Z)\] and use $h$ to consider $\H_3$ as a subgroup of $\U$ which satisfies 
\[\H_3\cap \U_0=\L.\]

We note that $Y$ and $Z$ are central in $\lN_0$ hence they belong to $\lN'$ whenever $(\phi,\lN')\in \mathcal{P}(\lN_0)$. 

By \cite[Proof of Theorem 4]{VD} we have:

\begin{prop}\label{prop of VD}
 Let $\pi$ be an irreducible representation of $\U$ with non trivial central character $c_\pi$, then there is $\pi_0\in \Irr(\U_0)$ such that \[\pi=\ind_{\U_0}^{\U}(\pi_0).\] 
  In fact one can choose $\pi_0$ such that if we identify the space of $\pi$ with $\mathcal{C}_c^\infty(F,V_{\pi_0})$ via the map 
$f\mapsto [x\mapsto f(h(x,0,0))]$, setting $\chi(z)=c_\pi(h(0,0,z))$, we have 
\begin{equation}\label{equation action of pi0} (\pi(u_0)f)(0)=\pi_0(u_0)f(0)\end{equation} for any $u_0\in U_0$ and 
\begin{equation}\label{equation action de H3} (\pi(h(x,y,z))f)(x')=\chi(z+x'y)f(x'+x).\end{equation} 
\end{prop}

Note that Equation (\ref{equation action of pi0}) is automaticially satisfied when $\pi=\ind_{\U}^{\U_0}(\pi_0)$. On the other hand 
Equation (\ref{equation action de H3}) is not. One can in fact characterize the representations $\pi_0$ of $\U_0$ in the above proposition:

\begin{lemma}\label{lemma complement of VD}
The irreducible representation $\pi_0$ is such that Equation (\ref{equation action de H3}) is satisfied if and only if $c_{\pi_0}$ is trivial on $h(0,F,0)$.
\end{lemma}
\begin{proof}
Suppose that Equation (\ref{equation action de H3}) is satisfied. Then by Equation (\ref{equation action of pi0}) and Equation (\ref{equation action de H3}) evaluated at $x'=x=z=0$, we see that the group $h(0,E,0)$ acts trivially 
on $V_{\pi_0}$. Conversely, suppose that $h(0,E,0)$ acts trivially on $V_{\pi_0}$. Then 
\[(\pi(h(x,y,z))f)(x')=\pi(h(x',0,0)h(x,y,z))f)(0)=\pi(h(x+x',y,z+x'y))f)(0)\]
\[= \pi(h(0,y,z+yx')h(x+x',0,0))f)(0)=\pi_0(h(0,y,z+x'y))(\pi(h(x+x',0, 0))f)(0)\]
\[=\chi(z+x'y)(\pi(h(x+x',0, 0))f)(0)=\chi(z+x'y)f(x+x').\]
\end{proof}

We will say that $\pi_0\in \Irr(\U_0)$ as in Lemma \ref{lemma complement of VD} is \textit{good}. 

\subsection{Classification}

An immediate corollary of Proposition \ref{prop of VD} proved in \cite{VD} is:

\begin{corollary}\label{cor unitaire et admissible}
Any $\pi\in \Irr(\U)$ is admissible and unitary.
\end{corollary}
\begin{proof}
By induction on $\dim(\U)$. If $\dim(\U)=1$ it is clear. If not, if either $\dim(\Ze)\geq 2$ or if $c_\pi$ is trivial, 
then setting $\K=\Ker(c_\pi)$, the group $\overline{\U}=\U/\Ker(c_\pi)$ has dimension smaller than that of $\U$ and we conclude by induction because $\pi$ is a representation of $\overline{\U}$. If $\dim(\Ze)=1$ and $c_\pi$ is nontrivial we can write $\pi=\ind_{\U_0}^{\U}(\pi_0)$ with $\pi_0$ good, thanks to Proposition \ref{prop of VD}. In this case $\pi_0$ must be irreducible so by induction it is unitary and admissible, from which we already conclude that $\pi$ is unitary. Moreover take a function 
$f\in \ind_{\U_0}^{\U}(\pi_0)\simeq \mathcal{C}_c^\infty(F,V_{\pi_0})$ which is fixed by a compact open subgroup $L$ of $\U$. Then by Equation (\ref{equation action de H3}) there is $k\in \Z$ depending on $L$ only such that $f$ is an $\w_F^kO_F$-invariant function on $F$, and by Equation (\ref{equation action de H3}) it must vanish outside the orthogonal of $\w_F^kO_F$ with respect to $\chi$. Hence $f$ is determined by its values on a finite set $A$ depending on $L$ but not on $f\in \pi^L$, and moreover its 
image is a subset of the finite dimensional space 
$V_{\pi_0}^{L'}$ where $L'=\cap_{a\in A} a^{-1}L a$. This means that $\ind_{\U_0}^{\U}(\pi_0)^L$ has finite dimension so that $\pi$ is admissible.
\end{proof}

 Because irreducible representations are unitary the following 
can be proved. 

\begin{corollary}\label{cor induction irreductible}
Suppose that $Z$ has dimension $1$, and let $\pi_0\in \Irr(\U_0)$ be a good representation, then $\pi=\ind_{\U_0}^{\U}(\pi_0)$ is irreducible. Moreover if $\pi_0'\in \Irr(\U_0)$ is another good representation such that $\pi=\ind_{\U_0}^{\U}(\pi_0')$, then $\pi_0'\simeq \pi_0$.
\end{corollary}
\begin{proof}
Because $\pi_0$ is unitary so is $\pi$, hence $\pi$ is semi-simple, and it is thus sufficient to prove that 
$\Hom_{\U}(\pi,\pi)$ is one dimensional. Now Equation (\ref{equation action de H3}) is satisfied for $\pi_0$ and $\pi$ thanks to our hypothesis, and the proof of Corollary \ref{cor unitaire et admissible} shows that 
$\pi$ is in fact admissible, so 
\[\pi=\Ind_{\U_0}^{\U}(\pi_0)\simeq\Ind_{\U_0}^{\U}(\pi_0').\] Hence one has 
\[\Hom_{\U}(\pi,\pi)\simeq \Hom_{\U_0}(\pi,\pi_0')\] and it remains to show that this latter space is one dimensional when $\pi_0'\simeq \pi_0$ and $\{0\}$ otherwise. Take $L\in\Hom_{\U_0}(\pi,\pi_0')$. We identify $\pi$ with $\mathcal{C}_c^\infty(F,V_{\pi_0}).$ For $\phi \in \mathcal{C}_c(F)$ and 
$f\in \mathcal{C}_c^\infty(F,V_{\pi_0})$ we set 
\[\pi(\phi)f=\int_{F} \phi(y)\pi(0,y,0)f dy.\] Note that \[(\pi(\phi)f)(x)=\widehat{\phi}(x)f(x)\] where the Fourier transform is taken with respect to $\psi$ and the fixed Haar measure on $F$.
On the other hand because $c_{\pi_0}(h(0,F,0)=\{1\}$ there is $c>0$ such that 
\[L(\pi(\phi) f)=\pi_0'(\phi) L(f)=c\widehat{\phi}(0)L(f)\] giving the equality 
\[L(\widehat{\phi}f)=c\widehat{\phi}(0)L(f)\] for all $\phi\in \mathcal{C}_c^\infty(F)$ and 
$f\in \mathcal{C}_c^\infty(F,V_{\pi_0}).$ In particular if $f(0)=0$, taking $\widehat{\phi}$ the characteristic function of a small enough compact open subgroup of $F$, we see that $L(f)=0$. This implies that there exists $L_0\in \Hom_{\U_0}(V_{\pi_0},V_{\pi_0'})$ such that \[L=[\phi\mapsto \phi(0)]\otimes L_0.\] We thus just exhibited a linear injection $L\mapsto L_0$ of 
$\Hom_{\U_0}(\pi,\pi_0)$ into $\Hom_{\U_0}(V_{\pi_0},V_{\pi_0}')$ which is zero if $\pi_0'\not \simeq \pi_0$ and one-dimensional by Schur's lemma otherwise. This concludes the proof.
\end{proof}

Before we state Kirillov's classification let's state another lemma.

\begin{lemma}\label{lemma reducible when non polarized}
Let $(\phi,\lN')$ be a pair where $\phi$ is a linear form on $\lN$ and $\lN'$ is a sub-algebra of $\lN$, such that $B_{\phi}$ is isotropic on $\lN'$, but which is not polarized, then 
$\pi(\U',\U,\psi_\phi)$ is reducible.
\end{lemma}
\begin{proof}
By transitivity of induction and because reducible representations induce to reducible ones, it is enough to show this when $(\phi,\lN)$ is polarized. In this case $\psi_\phi$ defines a character of the whole group $\U$. Suppose that $\ind_{\U'}^{\U}(\psi_\phi)$ was irreducible, in particular we would have 
 $\ind_{\U'}^{\U}(\psi_\phi)=\Ind_{\U'}^{\U}(\psi_\phi)$ by admissibility of irreducible representations. 
 But then \[\Hom_{\U}(\psi_\phi,\Ind_{\U'}^{\U}(\psi_\phi))\simeq \Hom_{\U'}(\psi_\phi,\psi_\phi)\neq 0\] which is absurd as it would imply that $\Ind_{\U'}^{\U}(\psi_\phi)$ is a character, which it is not by assumption.
\end{proof}

We can now obtain Kirillov's classification. 

\begin{thm}\label{theorem K}

\begin{enumerate}[1)]
\item Let $(\lN',\phi)$ be a pair consisting of a subalgebra of $\lN$ and a linear form $\phi$ on $\lN$ such that $\lN'$ is isotropic for $B_{\phi}$. The representation $\pi(\U',\U,\psi_\phi)$ is irreducible if and only if $(\phi,\lN')$ is polarized. 
\item Any irreducible representation of $\U$ is of the form $\pi(\U',\U,\psi_\phi)$ with $(\phi,\lN')$ polarized.
\item Two irreducible representations $\pi(\U',\U,\psi_\phi)$ and $\pi(\U'',\U,\psi_{\phi'})$ are 
isomorphic if and only if $\phi$ and 
$\phi'$ are in the same $\U$-orbit.
\end{enumerate}
\end{thm}
\begin{proof}
Thanks to Lemma \ref{lemma reducible when non polarized}, the first point will be proved if we show that 
$\pi(\U',\U,\psi_\phi)$ is irreducible when $(\phi,\lN')$ is polarized. We do an induction on $\dim(\U)$. If it is $1$ there is nothing to prove. If not we take $\pi=\pi(\U',\U,\psi_\phi)$ with $(\phi,\lN')\in\mathcal{P}(\lN)$. If $\dim(\U)>1$ and $c_\pi$ is trivial or if $\dim(\Ze)>1$, take $H\in \lZ$ such that $\phi(H)=1$, then 
$\pi$ is in fact a representation of $\U/\exp(F.H)$, and $(\lN'/F.H,\overline{\phi})\in \mathcal{P}(
\lN/F.H)$, so we conclude by induction. If $\dim(\Ze)=1$ and $c_\pi$ is non trivial, then according to \cite[Lemma 5.1]{K} we can suppose that $\lN'$ is a subalgebra of $\lN_0$ and that $\phi(Y)=0$. Then the pair $(\phi_{|\lN_0},\lN')$ is polarized and by induction the representation $\pi_0=\pi(\U',\U_0,\psi_{\phi_{|\lN_0}})$ is irreducible; it is moreover good because $\phi(Y)=0$. But then $\pi(\U',\U,\psi_\phi)=\ind_{\U_0}^{\U}(\pi_0)$ is irreducible thanks to Corollary \ref{cor induction irreductible}.

For point 2) we do again an induction on $\dim(\U)$, the one dimension case being obvious. Then if $c_\pi$ is trivial or if $\dim(\Ze)>1$ we 
conclude by induction. If not $\pi=\ind_{\U_0}^{\U}(\pi_0)$ with $\pi_0$ good. By induction 
$\pi_0=\ind_{\U'}^{\U_0}(\psi_{\phi_0})$ for $(\phi_0,\lN')\in \mathcal{P}(\lN_0)$. Then extend $\phi_0$ to a linear 
form $\phi$ on $\lN=F.X\oplus \lN_0$, we claim that the pair $(\phi,\lN')$ remains polarized. Indeed if it was not 
then one would have $\phi([X',\lN'])=0$ for $X'\notin \lN_0$. Writing $X'=aX+N_0$ with $N_0\in \lN_0$, then in particular
one would have $\phi([aX+N_0,Y])=0$, but $[aX+N_0,Y]=aZ+0=aZ$ so this would mean that $\phi(Z)=0$ i.e. that $c_\pi$ is trivial, which it is not.

Point 3) is proved by induction on $\dim(\U)$ as well, and we only focus on the case $\dim(\Ze)=1$ and $c_{\pi}\neq \1$. By \cite[Lemma 5.1]{K} we can suppose that 
both $\lN''$ and $\lN'$ are sub-algebras of $\lN_0$ and that $\phi(Y)=\phi'(Y)=0$. In particular 
$\pi_0=\pi(\U',\U_0,\psi_{\phi_{|\lN_0}})$ and $\pi_0'=\pi(\U'',\U_0,\psi_{\phi'_{|\lN_0}})$are both good, and both induce to $\pi$ so they are isomorphic by Corollary \ref{cor induction irreductible}. By induction this means that 
$\phi_{|\lN_0}$ and $\phi'_{|\lN_0}$ are $\U_0$-conjugate. Then it is explained just before \cite[Lemma 5.2]{K} at the end of the proof of \cite[Theorem 5.2]{K} that this implies that $\phi$ and $\phi'$ are indeed $\U$-conjugate.
\end{proof}

\begin{notation}
By the third point of Theorem \ref{theorem K}, the isomorphism class of the representation $\pi(\U',\U,\phi)$ only depends on 
$\phi$, we set \[\pi(\psi_\phi):=\pi(\U',\U,\psi_\phi).\]
\end{notation}

\section{Unipotent symmetric spaces}

We recall that the map $x\mapsto x^2$ is bijective from $\U$ to itself. We set $\U^{\sigma,-}$ for the closed subset of $\U$ given by the equation $\sigma(u)=u^{-1}$. We have a polar decomposition on $\U$.

\begin{lemma}\label{lemme polar dec}
The multiplication map $m:\U^{\sigma}\times \U^{\sigma,-}\rightarrow \U$ given by $m(u^+,u^-)=u^+u^-$ is a homeomorphism.
\end{lemma}
\begin{proof} This is is just \cite[Proposition 2.1, 3)]{B1}, the proof of which is valid in our setting. 
\end{proof}

We will use the following fixed point result in replacement of that used in \cite[Proof of Lemma 4.3.1]{B1}. 
It could be used in ibid. as well.

\begin{lemma}\label{lemma pro p action with involution}
Let $X$ be the $F$-points of an $F$-algebraic variety on which $\U$ acts in an $F$-rational manner, and $\sigma$ be an $F$-rational involution of $X$ 
(i.e. we have two involutions on different sets which we denote by the same letter) such that $\sigma(u.x)=\sigma(u).\sigma(x)$ for all $u\in \U$ and $x\in X$.
Then a $\U$-orbit in $X$ is $\sigma$-stable if and only if it contains a fixed point of $\sigma$. 
\end{lemma}
\begin{proof}
Take $O=\U.x$ a $\U$-orbit in $X$. If it contains a $\sigma$-fixed point $y$, then $y=u.x$ and 
$\sigma(x)=\sigma(u).y=\sigma(u)u^{-1}.x$ so $O$ is $\sigma$-stable. Conversely suppose that $\sigma(O)=O$. We denote by $\K$ the stabilizer of $x$ (it is an $F$-subgroup of $\U$). If 
$\K=\U$ there is nothing to prove. If not because $\mathbf{U}$ is unipotent there is a sequence $\K<\V\triangleleft\U$ with $\V$ a normal $F$-subgroup of $\U$ such that $\U/\V$ is commutative of 
dimension $1$ (this property can be proved by induction on $(\dim(\U),\dim(\U)-\dim(\K))$ with lexicographic ordering). Now $\sigma(x)=u.x$ for $u\in \U$ by assumption. This implies that $\sigma(u)u$ belongs to 
$\K$ hence to $\V$, so $\overline{\sigma(u)}=\overline{u}^{-1}\in \U/\V$, i.e. 
\[\overline{u^+}\ {\overline{u^-}}^{-1}={\overline{u^-}}^{-1}{\overline{u^+}}^{-1}={\overline{u^+}}^{-1}{\overline{u^-}}^{-1}\Leftrightarrow \overline{u^+}^2=\overline{1}\] in $\U/\V$, which implies that $\overline{u^+}=\overline{1}\in \U/\V$, so that $u^+\in \V$. However because 
$u^+$ is fixed by $\sigma$, it implies that \[u^+\in \V\cap \sigma(\V).\] Note that because $\V$ is normal in $\U$ so is 
$\sigma(\V)$ hence \[\V\cap \sigma(\V)\triangleleft\U.\] We set \[u_1=(u^-)^{1/2}\] so that 
\[\sigma(u_1)=u_1^{-1}\] (because this relation is true when squared) and 
\[v=u_1^{-1}u u_1^{-1}=u_1^{-1} u^+ u_1,\] hence \[v\in \V\cap \sigma(\V).\]
So setting \[y=u_1.x\] this implies that 
\[\sigma(y)=\sigma(u_1).\sigma(x)=u_1^{-1}u.x=v.y\] Hence 
$\sigma(y)$ and $y$ are in the same $\V\cap \sigma(\V)$-orbit $O'$ inside $O$. Now because $\V\cap \sigma(\V)$ is $\sigma$-stable and has dimension smaller than that of $\U$, we conclude by induction that $\sigma$ fixes a point of $O'$, hence a point of $O$.
\end{proof}

 We make $\sigma$ act on $\lN^*$ by the formula 
 \[\sigma(\phi)=-\phi^\sigma.\] Then a very special case of Lemma \ref{lemma pro p action with involution} is:

\begin{lemma}\label{lemme existence de polarisations stables}
Take $\phi\in \lN^*$, then $\sigma(\phi)$ and $\phi$ are in the same $\U$-orbit if and only if there is a $\sigma$-fixed linear form in the $\U$-orbit of $\phi$, i.e. a linear form which vanishes on $\lN^\sigma$.
\end{lemma}

Finally we have:

\begin{lemma}\label{lemma orbits of distinguished polarizations}
Let $\phi$  and $\phi'$ be two $\sigma$-fixed $\U$-conjugate linear forms on $\lN$, then they are $\U^\sigma$-conjugate.  
\end{lemma}
\begin{proof}
It follows from the polar decomposition as in \cite[Lemma 4.3.1, b)]{B1}. Note that in the proof of ibid. it is enough to argue that if $u^2$ is in the stabilizer of $\phi$, then clearly $u$ is because the stabilizer in question is unipotent as well (so that $u \mapsto u^2$ is a bijection of it).
\end{proof}

\section{Distinction, conjugate self-duality and multiplicity one}\label{section distinction}

We now recover the results we are interested in from \cite{B1} and \cite{B2}, with the same proofs. Multiplicity one and conjugate self-duality for distinguished representations of $\U$ follow from the Gelfand-Kazhdan argument, or more precisely its simplification by Bernstein-Zelevinsky (\cite{BZ76}). We indeed notice that 
the space of double cosets \[\U^{\sigma}\backslash \U /\U^{\sigma}\] 
is fixed by the anti-involution 
\[\theta(g)\rightarrow \sigma(g)^{-1}\] thanks to Lemma \ref{lemme polar dec}. In particular any 
bi-$\U^\sigma$-invariant distribution on $\U$ is fixed by $\theta$ thanks to \cite[Theorems 6.13 and 6.15]{BZ76}. This implies 
as in \cite{GK}, or more precisely as in \cite[Lemma 4.2]{Ptri}, that for any irreducible representation $\pi\in \Irr(\U)$ one has 
\[\dim(\Hom_{\U^\sigma}(\pi,\C))\dim(\Hom_{\U^\sigma}(\pi^\vee,\C))\leq 1.\]

\begin{prop}\label{prop mult 1}
For $\pi\in \Irr_{\U^\sigma}(\U)$ one has $\dim(\Hom_{\U^\sigma}(\pi,\C))\leq 1$ and 
$\pi^\vee\simeq \pi^\sigma$. 
\end{prop}
\begin{proof}
Suppose that $\pi$ is distinguished and take $L\in \Hom_{\U^\sigma}(\pi,\C)-\{0\}$. Because $\pi$ is unitary its contragredient $\pi^\vee$ it is isomorphic to $\overline{\pi}$ where $\overline{\pi}=c\circ \pi\circ c^{-1}$ with $c$ the complex conjugation on the space of $\pi$ obtained by the choice of a basis of this space. In particular $\overline{L}=L\circ c^{-1}\in \Hom_{\U^\sigma}(\overline{\pi},\C)$.
 Then the map \[T_{L,\overline{L}}:f\in \mathcal{C}_c^\infty(\U)\mapsto \overline{L}((\pi(f) L)\] is a bi-$\U^\sigma$-invariant hence fixed by $\theta$. We conclude by applying \cite[Lemma 3]{H} (where we take $\H_1=\H_2=\U^{\sigma}$ and 
 $\chi_2(zu^+)=\chi_1(zu^+)=c_\pi(z)$ for $u^+\in \U^{\sigma}$ and $z\in \Ze$, remembering that $c_\pi$ is necessarily trivial on $\Ze^\sigma$).
\end{proof}

Note that $(\lN^*)^\sigma$ and $(\frac{\lN}{\lN^\sigma})^*$ are canonically isomorphic, and we identify them. It is a space acted upon by $\U^\sigma$. Before stating the main theorem, we recall \cite[Lemma 2.2.1]{B2}, the proof of which is valid over $F$ (as it relies on \cite[Proposition 1.1.2]{V} which has no assumption on the field).

\begin{lemma}\label{polarisation stable}
Take $\phi\in(\frac{\lN}{\lN^\sigma})^*$, then there is a $\sigma$-stable Lie sub-algebra $\lN'$ of $\lN$ such that $(\phi,\lN')$ is polarized.
\end{lemma}

We can now prove the following result.

\begin{thm}\label{thm distinction and sel-duality}
A representation $\pi\in \Irr_{\U^\sigma}(\U)$ is distinguished if and only $\pi^\vee=\pi^\sigma$. Moreover the map $\U^{\sigma}.\phi\mapsto \pi(\psi_\phi)$ is a bijection from 
$\U^\sigma\backslash (\frac{\lN}{\lN^\sigma})^*$ to $\Irr_{\U^\sigma \dist}(\U)$.
\end{thm}
\begin{proof}
Suppose that $\pi=\pi(\psi_{\phi})\in \Irr(\U)$ is conjugate self-dual, then $\sigma (\phi)$ and 
$\phi'$ are in the same $\U$-orbit, which must contain a $\sigma$-fixed linear form thanks to Lemma \ref{lemme existence de polarisations stables}. So we can in fact suppose that $\phi\in (\frac{\lN}{\lN^\sigma})^*$. In particular by Lemma \ref{polarisation stable} we can write $\pi(\psi_\phi)=\pi(\U',\U,\psi_\phi)$ for $\U'=\exp(\lN')$ which is $\sigma$-stable. The quotient $\U'^\sigma\backslash \U^\sigma$ identifies with a closed subset of $\U'\backslash \U$ and the condition 
$\phi\in(\frac{\lN}{\lN^\sigma})^*$ implies that $\psi_\phi$ is trivial on $\U'^\sigma$. 
 Then $\pi$ is distinguished, with explicit linear nonzero $\U^\sigma$-invariant linear form given on $\pi$ by 
\[\lambda:f\mapsto \int_{\U'^\sigma\backslash \U^\sigma} f(u)du.\] To finish the proof it remains to prove the injectivity 
of the map $\U^{\sigma}.\phi\mapsto \pi(\psi_{\phi})$, which is Lemma \ref{lemma orbits of distinguished polarizations}.
\end{proof}

In particular in the case of the Galois involution one gets a bijective correspondence between 
$\Irr(\U^\sigma)$ and $\Irr_{\U^\sigma \dist}(\U)$. Indeed $\mathbf{U}=\Res_{E/F}(\mathbf{U}^\sigma)$ for $E$ a quadratic extension of $F$. Writing $\d$ for an element of $E-F$ with square in $F$. One can identify the space $(\lN^\sigma)^*$ to the space $(\lN^*)^\sigma$ by the map 
\[\mathrm{C}: \phi_\sigma\rightarrow \phi\] where 
\[\phi(N+\d N')=\phi_\sigma(N').\] This yields:

\begin{corollary}\label{cor correspondence}
When $E/F$ is a Galois involution, the map $\pi(\psi_{\phi_\sigma})\rightarrow \pi(\psi_{\phi})$ is 
a bijective correspondence from $\Irr(\U^\sigma)$ to $\Irr_{\U^\sigma\dist}(\U)$
\end{corollary}

\textbf{Acknowledgements}. We thank Dipendra Prasad for useful comments on a previous version of this note concerned only with Galois involutions,
and Abderrazak Bouaziz for bringing the papers \cite{B1} and \cite{B2} to our attention, which lead to the actual version of this note. We thank Maarten Solleveld for numerous useful 
comments and for pointing out a mistake in a previous version of Lemma \ref{lemma pro p action with involution}, and Ahmed Moussaoui for his help in finding a reference. We thank Pierre Torasso for pointing out a persisting mistake in Lemma \ref{lemma pro p action with involution}. We thank Martin Andler for clarifying a misunderstanding of the author concerning the Kirillov classification.

\bibliographystyle{plain}
\bibliography{unipotentdistinction}

\end{document}